\documentclass[a4paper,leqno]{amsart}

\usepackage{amsmath}
\usepackage{amsthm}
\usepackage{amssymb}
\usepackage{hyperref}
\usepackage{caption}
\usepackage{subcaption}
\usepackage{graphicx}
\usepackage{lineno}

\newtheorem{theorem}{Theorem}
\newtheorem{maintheorem}{Main Theorem}

\newtheorem{proposition}[theorem]{Proposition}

\newtheorem{corollary}[theorem]{Corollary}
\newtheorem{definition}[theorem]{Definition}

\newtheorem{algorithm}[theorem]{Algorithm}

\numberwithin{equation}{section}
\numberwithin{theorem}{section}

\newcommand{\abs}[1]{|#1|}
\newcommand{\set}[1]{\left\{#1\right\}}

\newcommand{\arrowsv}[0]{\overset{v}{\rightarrow}}

\newcommand{\uni}[2]{{#1}\big\vert_{#2}}

\newcommand{\wH}[3]{\widetilde{\mathcal{H}}(\uni{#1}{#2}; {#3})}
\newcommand{\wHn}[4]{\widetilde{\mathcal{H}}(\uni{#1}{#2}; {#3}; {#4})}
\newcommand{\wFv}[3]{\widetilde{F}_v(\uni{#1}{#2}; {#3})}

\DeclareMathOperator{\V}{V}
\DeclareMathOperator{\E}{E}
\DeclareMathOperator{\N}{N}

\everymath{\displaystyle}
\begin{document}
	
\title[Modified vertex Folkman numbers]
{Modified vertex Folkman numbers}

\author{Aleksandar Bikov}



\author{Nedyalko Nenov}






\dedicatory{}

	
\begin{abstract}
Let $a_1, ..., a_s$ be positive integers. For a graph $G$ the expression
$$
G \overset{v}{\rightarrow} (a_1, ..., a_s)
$$
means that for every coloring of the vertices of $G$ in $s$ colors ($s$-coloring) there exists $i \in \{1, ..., s\}$, such that there is a monochromatic $a_i$-clique of color $i$. If $m$ and $p$ are positive integers, then
$$
G \overset{v}{\rightarrow} {m}\big\vert_{p}
$$
means that for arbitrary positive integers $a_1, ..., a_s$ ($s$ is not fixed), such that $\sum_{i = 1}^{s}(a_i - 1) + 1 = m$ an $\max{\{a_1, ..., a_s\}} \leq p$ we have $G \overset{v}{\rightarrow} (a_1, ..., a_s)$. Let
$$
\widetilde{\mathcal{H}}({m}\big\vert_{p}; q) = \{G : G \overset{v}{\rightarrow} {m}\big\vert_{p} \mbox{ and } \omega(G) < q\}.
$$

The modified vertex Folkman numbers are defined by the equality
$$
\widetilde{F}({m}\big\vert_{p}; q) = \min{\{|V(G)| : G \in \widetilde{\mathcal{H}}({m}\big\vert_{p}; q)\}}.
$$
If $q \geq m$ these numbers are known and they are easy to compute. In the case $q = m - 1$ we know all of the numbers when $p \leq 5$. In this work we consider the next unknown case $p = 6$ and we prove with the help of a computer that
$$
\widetilde{F}({m}\big\vert_{6}; m - 1) = m + 10.
$$
\end{abstract}


\maketitle



\section{Introduction}

In this paper only finite, non-oriented graphs without loops and multiple edges are considered. The following notations are used:

$\V(G)$ - the vertex set of $G$;

$\E(G)$ - the edge set of $G$;

$\overline{G}$ - the complement of $G$;

$\omega(G)$ - the clique number of $G$;

$\alpha(G)$ - the independence number of $G$;

$\chi(G)$ - the chromatic number of $G$;

$\N(v), \N_G(v), v \in \V(G)$ - the set of all vertices of G adjacent to $v$;

$d(v), v \in \V(G)$ - the degree of the vertex $v$, i.e. $d(v) = \abs{\N(v)}$;

$G-v, v \in \V(G)$ - subgraph of $G$ obtained from $G$ by deleting the vertex $v$ and all edges incident to $v$;

$G-e, e \in \E(G)$ - subgraph of $G$ obtained from $G$ by deleting the edge $e$;

$G+e, e \in \E(\overline{G})$ - supergraph of G obtained by adding the edge $e$ to $\E(G)$.

$K_n$ - complete graph on $n$ vertices;

$C_n$ - simple cycle on $n$ vertices;

$m_0 = m_0(p)$ - see Theorem \ref{theorem: m_0};

$G_1 + G_2$ - a graph $G$ for which: $\V(G) = \V(G_1) \cup \V(G_2)$ and $\E(G) = \E(G_1) \cup \E(G_2) \cup E'$, where $E' = \set{[x, y] : x \in \V(G_1), y \in \V(G_2)}$, i.e. $G$ is obtained by connecting every vertex of $G_1$ to every vertex of $G_2$.

All undefined terms can be found in \cite{W01}.\\

Let $a_1, ..., a_s$ be positive integers. The expression $G \overset{v}{\rightarrow} (a_1, ..., a_s)$ means that for any coloring of $\V(G)$ in $s$ colors ($s$-coloring) there exists $i \in \set{1, ..., s}$ such that there is a monochromatic $a_i$-clique of color $i$. In particular, $G \overset{v}{\rightarrow} (a_1)$ means that $\omega(G) \geq a_1$.

Define:

$\mathcal{H}(a_1, ..., a_s; q) = \set{ G : G \overset{v}{\rightarrow} (a_1, ..., a_s) \mbox{ and } \omega(G) < q }.$

$\mathcal{H}(a_1, ..., a_s; q; n) = \set{ G : G \in \mathcal{H}(a_1, ..., a_s; q) \mbox{ and } \abs{\V(G)} = n }.$\\

The vertex Folkman number $F_v(a_1, ..., a_s; q)$ is defined by the equality:

\begin{equation*}
F_v(a_1, ..., a_s; q) = \min\set{\abs{\V(G)} : G \in \mathcal{H}(a_1, ..., a_s; q)}.
\end{equation*}

Folkman proves in \cite{Fol70} that:
\begin{equation}
\label{equation: F_v(a_1, ..., a_s; q) exists}
F_v(a_1, ..., a_s; q) \mbox{ exists } \Leftrightarrow q > \max\set{a_1, ..., a_s}.
\end{equation}
Other proofs of (\ref{equation: F_v(a_1, ..., a_s; q) exists}) are given in \cite{DR08} and \cite{LRU01}.\\
In \cite{LU96} for arbitrary positive integers $a_1, ..., a_s$ the following are defined
\begin{equation}
\label{equation: m and p}
m(a_1, ..., a_s) = m = \sum\limits_{i=1}^s (a_i - 1) + 1 \quad \mbox{ and } \quad p = \max\set{a_1, ..., a_s}.
\end{equation}
Obviously, $K_m \overset{v}{\rightarrow} (a_1, ..., a_s)$ and $K_{m - 1} \overset{v}{\nrightarrow} (a_1, ..., a_s)$. Therefore,
\begin{equation*}
F_v(a_1, ..., a_s; q) = m, \quad q \geq m + 1.
\end{equation*}
The following theorem for the numbers $F_v(a_1, ..., a_s; m)$ is true:
\begin{theorem}
\label{theorem: F_v(a_1, ..., a_s; m) = m + p}
Let $a_1, ..., a_s$ be positive integers and $m$ and $p$ are defined by (\ref{equation: m and p}). If $m \geq p + 1$, then:
\begin{flalign*}
F_v(a_1, ..., a_s; m) = m + p, \ \mbox{\cite{LU96},\cite{LRU01}}. && \tag{a} 
\end{flalign*}
\begin{flalign*}
K_{m+p} - C_{2p + 1} = K_{m - p - 1} + \overline{C}_{2p + 1} && \tag{b}
\end{flalign*}
is the only extremal graph in $\mathcal{H}(a_1, ..., a_s; m)$, \ \cite{LRU01}.
\end{theorem}
The condition $m \geq p + 1$ is necessary according to (\ref{equation: F_v(a_1, ..., a_s; q) exists}). Other proofs of Theorem \ref{theorem: F_v(a_1, ..., a_s; m) = m + p} are given in \cite{Nen00} and \cite{Nen01}.\\

Very little is known about the numbers $F_v(a_1, ..., a_s; q)$, $q \leq m - 1$. In this work we suggest a method to bound these numbers with the help of the modified vertex Folkman numbers $\wFv{m}{p}{q}$, which are defined below.

\begin{definition}
\label{definition: G overset(v)(rightarrow) uni(m)(p)}
Let $G$ be a graph and $m$ and $p$ be positive integers. The expression
\begin{equation*}
G \overset{v}{\rightarrow} \uni{m}{p}
\end{equation*}
means that for any choice of positive integers $a_1, ..., a_s$ ($s$ is not fixed), such that $m = \sum\limits_{i=1}^s (a_i - 1) + 1$ and $\max\set{a_1, ..., a_s} \leq p$, we have
\begin{equation*}
G \overset{v}{\rightarrow} (a_1, ..., a_s).
\end{equation*}
\end{definition}

Define:

$\wH{m}{p}{q} = \set{G : G \overset{v}{\rightarrow} \uni{m}{p} \mbox{ and } \omega(G) < q}$.

$\wHn{m}{p}{q}{n} = \set{G : G \in \wH{m}{p}{q} \mbox{ and } \abs{\V(G)} = n}$.\\

The modified vertex Folkman numbers are defined by the equality:
\begin{equation*}
\wFv{m}{p}{q} = \min\set{\abs{\V(G)} : G \in \wH{m}{p}{q}}.
\end{equation*}

The graph $G$ is called a maximal graph in $\wH{m}{p}{q}$ if $G \in \wH{m}{p}{q}$, but $G + e \not\in \wH{m}{p}{q}, \forall e \in \E(\overline{G})$, i.e. $\omega(G + e) \geq q, \forall e \in \E(\overline{G})$.

For convenience we will also define the following term:
\begin{definition}
\label{definition: (+K_t)}
The graph $G$ is called a $(+K_t)$-graph if $G + e$ contains a new $t$-clique for all $e \in \E(\overline{G})$.
\end{definition}
Obviously, $G \in \wH{m}{p}{q}$ is a maximal graph in $\wH{m}{p}{q}$ if and only if $G$ is a $(+K_q)$-graph.\\

From the definition of the modified Folkman numbers it becomes clear that if $a_1, ..., a_s$ are positive integers and $m$ and $p$ are defined by (\ref{equation: m and p}), then 
\begin{equation}
\label{equation: F_v(a_1, ..., a_s; q) leq wFv(m)(p)(q)}
F_v(a_1, ..., a_s; q) \leq \wFv{m}{p}{q}.
\end{equation}
Defining and computing the modified Folkman numbers is appropriate because of the following reasons:

1) On the left side of (\ref{equation: F_v(a_1, ..., a_s; q) leq wFv(m)(p)(q)}) there is actually a whole class of numbers, which are bound by only one number $\wFv{m}{p}{q}$.

2) The upper bound for $\wFv{m}{p}{q}$ is easier to compute than the numbers $F_v(a_1, ..., a_s)$ because of the following
\begin{theorem}
\label{theorem: wFv(m)(p)(m - m_0 + q) leq wFv(m_0)(p)(q) + m - m_0}
(\cite{BN15}, Theorem 7.2) Let $m$, $m_0$, $p$ and $q$ be positive integers, $m \geq m_0$ and $q > \min\set{m_0, p}$. Then
\begin{equation*}
\wFv{m}{p}{m - m_0 + q} \leq \wFv{m_0}{p}{q} + m - m_0.
\end{equation*}
\end{theorem}
Therefore, if we know the value of one number $\wFv{m'}{p}{q}$ we can obtain an upper bound for $\wFv{m}{p}{q}$ where $m \geq m'$.

3) As we will see below (Theorem \ref{theorem: m_0}), the computation of the numbers $\wFv{m}{p}{m - 1}$ is reduced to finding the exact values of the first several of these numbers (bounds for the number of exact values needed are given in \ref{theorem: m_0} (c)).\\

Let $A$ be an independent set of vertices in $G$. If $V_1 \cup ... \cup V_s$ is $(a_1, ..., a_s)$-free $s$-coloring of $\V(G - A)$ (i.e. $V_i$ does not contain an $a_i$-clique, $i = 1, ..., s$), then $A \cup V_1 \cup ... \cup V_s$ is $(2, a_1, ... , a_s)$-free $(s + 1)$-coloring of $\V(G)$. Therefore
\begin{equation}
\label{equation: G arrowsv (2, a_1, ..., a_s) Rightarrow G - A arrowsv (a_1, ..., a_s)}
G \arrowsv (2, a_1, ..., a_s) \Rightarrow G - A \arrowsv (a_1, ..., a_s).
\end{equation}
Further we will need the following
\begin{proposition}
\label{proposition: G - A arrowsv uni((m - 1))(p)}
Let $G \arrowsv \uni{m}{p}$ and $A$ is an independent set of vertices in $G$. Then $G - A \arrowsv \uni{(m - 1)}{p}$.
\end{proposition}

\begin{proof}
Let $a_1, .., a_s$ be positive integers, such that
\begin{equation*}
m - 1 = \sum\limits_{i=1}^s (a_i - 1) + 1 \ \mbox{ and } 2 \leq a_i \leq p.
\end{equation*}
Then
\begin{equation*}
m = (2 - 1) + \sum\limits_{i=1}^s (a_i - 1) + 1.
\end{equation*}
It follows that $G \arrowsv (2, a_1, ..., a_s)$ and from (\ref{equation: G arrowsv (2, a_1, ..., a_s) Rightarrow G - A arrowsv (a_1, ..., a_s)}) we obtain $G - A \arrowsv (a_1, ..., a_s)$.
\end{proof}

It is easy to see that if $q > m$, then $F_v(a_1, ..., a_s; q) = \wFv{m}{p}{q} = m$. From Theorem \ref{theorem: F_v(a_1, ..., a_s; m) = m + p} it follows that $F_v(a_1, ..., a_s; m) = \wFv{m}{p}{m} = m + p$. In the case $q = m - 1$ the following general bounds are known:
\begin{equation}
\label{equation: m + p + 2 leq wFv(m)(p)(m - 1) leq m + 3p, m geq p + 2}
m + p + 2 \leq \wFv{m}{p}{m - 1} \leq m + 3p, \ m \geq p + 2.
\end{equation}
The upper bound follows from the proof of the Main Theorem from \cite{KN06} and the lower bound follows from (\ref{equation: F_v(a_1, ..., a_s; q) leq wFv(m)(p)(q)}) and $F_v(a_1, ..., a_s; q) \geq m + p + 2$, \cite{Nen00}.

We know all the numbers $\wFv{m}{p}{m - 1}$ where $p \leq 5$ (in the cases $p \leq 4$ see the Remark after Theorem 4.5 and (1.5) from \cite{BN15}, and in the case $p = 5$ see Theorem 7.4 also from \cite{BN15}). It is also known that
\begin{equation*}
m + 9 \leq \wFv{m}{6}{m - 1} \leq m + 10, \ \mbox{\cite{BN15}}
\end{equation*}
In this work we complete the computation of the numbers $\wFv{m}{6}{m - 1}$ by proving
\begin{maintheorem}
$\wFv{m}{6}{m - 1} = m + 10, \ m \geq 8$.
\end{maintheorem}

\section{A theorem for the numbers $\wFv{m}{p}{m - 1}$}

We will need the following fact:
\begin{equation}
\label{equation: G overset(v)(rightarrow) (a_1, ..., a_s) Rightarrow chi(G) geq m}
G \overset{v}{\rightarrow} (a_1, ..., a_s) \Rightarrow \chi(G) \geq m, \ \mbox{\cite{Nen01} (see also \cite{BN15})}.
\end{equation}
It is easy to prove (see Proposition 4.4 from \cite{BN15}) that
\begin{equation}
\label{equation: wFv(m)(p)(m - 1) exists ...}
\wFv{m}{p}{m - 1} \mbox{ exists } \Leftrightarrow m \geq p + 2.
\end{equation}
In \cite{BN15}(version 1) we formulate without proof the following
\begin{theorem}
\label{theorem: m_0}
Let $m_0(p) = m_0$ be the smallest positive integer for which
\begin{equation*}
\min_{m \geq p + 2}\set{\wFv{m}{p}{m - 1} - m} = \wFv{m_0}{p}{m_0 - 1} - m_0.
\end{equation*}
Then:
\begin{flalign*}
\wFv{m}{p}{m - 1} = \wFv{m_0}{p}{m_0 - 1} + m - m_0, \quad m \geq m_0. && \tag{a}
\end{flalign*}
\begin{flalign*}
\mbox{if $m_0 > p + 2$ and $G$ is an extremal graph in $\wH{m_0}{p}{m_0 - 1}$, then } && \tag{b}
\end{flalign*}
$G \overset{v}{\rightarrow} (2, m_0 - 2)$.\\
\begin{flalign*}
m_0 < \wFv{(p + 2)}{p}{p + 1} - p. && \tag{c}
\end{flalign*}
\end{theorem}

In this section we present a proof of Theorem \ref{theorem: m_0}.

The condition $m \geq p + 2$ is necessary according to (\ref{equation: wFv(m)(p)(m - 1) exists ...}).

\begin{proof}
(a) According to the definition of $m_0(p) = m_0$ we have

$\wFv{m}{p}{m - 1} \geq \wFv{m_0}{p}{m_0 - 1} + m - m_0, \ m \geq p + 2.$\\
According to Theorem \ref{theorem: wFv(m)(p)(m - m_0 + q) leq wFv(m_0)(p)(q) + m - m_0} if $m \geq m_0$ the opposite inequality is also true.

(b) Assume the opposite is true and let

$\V(G) = V_1 \cup V_2, V_1 \cap V_2 = \emptyset$,\\
where $V_1$ is an independent set and $V_2$ does not contain an $(m_0 - 2)$-clique. Let $G_1 = G[V_2] = G - V_1$. According to Proposition \ref{proposition: G - A arrowsv uni((m - 1))(p)}, from $G \overset{v}{\rightarrow} \uni{m_0}{p}$ it follows $G_1 \overset{v}{\rightarrow} \uni{(m_0 - 1)}{p}$. Since $\omega(G_1) < m_0 - 2$, $G_1 \in \wH{(m_0 - 1)}{p}{m_0 - 2}$. Therefore

$\abs{\V(G)} - 1 \geq \abs{\V(G_1)} \geq \wFv{(m_0 - 1)}{p}{m_0 - 2}$.\\
Since $\abs{\V(G)} = \wFv{m_0}{p}{m_0 - 1}$, from these inequalities it follows that

$\wFv{m_0}{p}{m_0 - 1} - m_0 \geq \wFv{(m_0 - 1)}{p}{m_0 - 2} - (m_0 - 1)$,\\
which contradicts the definition of $m_0$.

(c) If $m_0 = p + 2$, then from (\ref{equation: m + p + 2 leq wFv(m)(p)(m - 1) leq m + 3p, m geq p + 2}) we have $\wFv{(p + 2)}{p}{p + 1} \geq 2p + 4 = p + 2 + m_0$ and therefore in this case the inequality (c) is true.

Let $m_0 > p + 2$ and $G$ be an extremal graph in $\wH{m_0}{p}{m_0 - 1}$. If $a_1, ..., a_s$ are positive integers, such that $m = \sum\limits_{i=1}^s (a_i - 1) + 1$ and $\max\set{a_1, ..., a_s} \leq p$, then $G \overset{v}{\rightarrow} (a_1, ..., a_s)$ and according to (\ref{equation: G overset(v)(rightarrow) (a_1, ..., a_s) Rightarrow chi(G) geq m}), $\chi(G) \geq m_0$. From (b) and Theorem \ref{theorem: F_v(a_1, ..., a_s; m) = m + p} we see that $\abs{\V(G)} \geq 2m_0 - 3$ and $\abs{\V(G)} = 2m_0 - 3$ only if $G = \overline{C}_{2m_0 - 3}$. However, the last equality is not possible because $\chi(G) \geq m_0$ and $\chi(\overline{C}_{2m_0 - 3}) = m_0 - 1$. Therefore

$\abs{\V(G)} = \wFv{m_0}{p}{m_0 - 1} \geq 2m_0 - 2$\\
Since $m_0 > p + 2$ from the definition of $m_0$ we have

$\wFv{m_0}{p}{m_0 - 1} - m_0 < \wFv{(p + 2)}{p}{p + 1} - p - 2$.\\
From these inequalities the inequality (c) follows easily.
\end{proof}

\section{Algorithms}

In this section we present algorithms for finding all maximal graphs in $\wHn{m}{p}{q}{n}$ with the help of a computer. The remaining graphs in this set can be obtained by removing edges from the maximal graphs. The idea for these algorithms comes from \cite{PRU99} (see Algorithm А1). Similar algorithms are used in \cite{BN15}, \cite{CR06}, \cite{XLS10}, \cite{LR11}, \cite{SLPX12}. Also with the help of the computer, results for Folkman numbers are obtained in \cite{JR95}, \cite{SXP09}, \cite{SXL09} and \cite{DLSX13}.\\
The following proposition for maximal graphs in $\wHn{m}{p}{q}{n}$ will be useful

\begin{proposition}
\label{proposition: finding all maximal graphs in wHn(m)(p)(q)(n)}
Let $G$ be a maximal graph in $\wHn{m}{p}{q}{n}$. Let $v_1, v_2, ..., v_k$ be independent vertices of $G$ and $H = G - \set{v_1, v_2, ..., v_k}$. Then:
\begin{flalign*}
H \in \wHn{(m - 1)}{p}{q}{n - k} && \tag{a}
\end{flalign*}
\begin{flalign*}
\mbox{$H$ is a $(+K_{q - 1})$-graph} && \tag{b}
\end{flalign*}
\begin{flalign*}
\mbox{$\N_G(v_i)$ is a maximal $K_{q - 1}$-free subset of $\V(H)$, $i = 1, ..., k$} && \tag{c}
\end{flalign*}
\end{proposition}

\begin{proof}
The proposition (a) follows from Proposition \ref{proposition: G - A arrowsv uni((m - 1))(p)}, (b) and (c) follow from the maximality of $G$.
\end{proof}

We will define an algorithm, which is based on Proposition \ref{proposition: finding all maximal graphs in wHn(m)(p)(q)(n)}, and generates all maximal graphs in $\wHn{m}{p}{q}{n}$ with independence number at least $k$. 

\begin{algorithm}
\label{algorithm: finding all maximal graphs in wHn(m)(p)(q)(n) with independence number at least k}
Finding all maximal graphs in $\wHn{m}{p}{q}{n}$ with independence number at least $k$ by adding $k$ independent vertices to the $(+K_{q - 1})$-graphs in $\wHn{(m - 1)}{p}{q}{n - k}$.
	
1. Denote by $\mathcal{A}$ the set of all $(+K_{q - 1})$-graphs in $\wHn{(m - 1)}{p}{q}{n - k}$. The obtained maximal graphs in $\wHn{m}{p}{q}{n}$ will be output in $\mathcal{B}$, let $\mathcal{B} = \emptyset$.
	
2. For each graph $H \in \mathcal{A}$:
	
2.1. Find the family $\mathcal{M}(H) = \set{M_1, ..., M_t}$ of all maximal $K_{q - 1}$-free subsets of $\V(H)$.
	
2.2. Consider all the $k$-tuples $(M_{i_1}, M_{i_2}, ..., M_{i_k})$ of elements of $\mathcal{M}(H)$, for which $1 \leq i_1 \leq ... \leq i_k \leq t$ (in these $k$-tuples some subsets $M_i$ can coincide). For every such $k$-tuple construct the graph $G = G(M_{i_1}, M_{i_2}, ..., M_{i_k})$ by adding to $\V(H)$ new independent vertices $v_1, v_2, ..., v_k$, so that $N_G(v_j) = M_{i_j}, j = 1, ..., k$ (see Proposition \ref{proposition: finding all maximal graphs in wHn(m)(p)(q)(n)} (c)). If $\omega(G + e) = q, \forall e \in \E(\overline{G})$, then add $G$ to $\mathcal{B}$.
	
3. Exclude the isomorph copies of graphs from $\mathcal{B}$.
	
4. Exclude from $\mathcal{B}$ all graphs which are not in $\wHn{m}{p}{q}{n}$.
\end{algorithm}

\begin{theorem}
\label{theorem: finding all maximal graphs in wHn(m)(p)(q)(n) with independence number at least k}
Upon completion of Algorithm \ref{algorithm: finding all maximal graphs in wHn(m)(p)(q)(n) with independence number at least k} the obtained set $\mathcal{B}$ is equal to the set of all maximal graphs in $\wHn{m}{p}{q}{n}$ with independence number at least $k$.
\end{theorem}

\begin{proof}
From step 4 we see that $\mathcal{B} \subseteq \wHn{m}{p}{q}{n}$ and from step 2.2 it becomes clear, that $\mathcal{B}$ contains only maximal graphs in $\wHn{m}{p}{q}{n}$ with independence number at least $k$. Let $G$ be an arbitrary maximal graph in $\wHn{m}{p}{q}{n}$ with independence number in $k$. We will prove that $G \in \mathcal{B}$. Let $v_1, ..., v_k$ be independent vertices of $G$ and $H = G - \set{v_1, ..., v_k}$. According to Proposition \ref{proposition: finding all maximal graphs in wHn(m)(p)(q)(n)}(a) and (b), $H \in \wHn{(m - 1)}{p}{q}{n - k}$ and $H$ is a $(+K_{q - 1})$-graph. Therefore in step 1 we have $H \in \mathcal{A}$. According to Proposition \ref{proposition: finding all maximal graphs in wHn(m)(p)(q)(n)}(c), $\N_{G}(v_i) \in \mathcal{M}(H)$ for all $i \in \set{1, ..., k}$, hence in step 2 $G$ is added to $\mathcal{B}$.
\end{proof}

Let us note that if $G \in \wHn{m}{p}{q}{n}$ and $n \geq q$, then $G \neq K_n$ and therefore $\alpha(G) \geq 2$. In this case, with the help of Algorithm \ref{algorithm: finding all maximal graphs in wHn(m)(p)(q)(n) with independence number at least k} we can obtain all maximal graphs in $\wHn{m}{p}{q}{n}$ by adding to independent vertices to the $(+K_{q - 1})$-graphs in $\wHn{(m - 1)}{p}{q}{n - 2}$. 

It is clear that if $G$ is a graph for which $\alpha(G) = 2$ and $H$ is a subgraph of $G$ obtained by removing independent vertices, then $\alpha(H) \leq 2$. We modify Algorithm \ref{algorithm: finding all maximal graphs in wHn(m)(p)(q)(n) with independence number at least k} in the following way to obtain the maximal graphs in $\wHn{m}{p}{q}{n}$ with independence number 2:
\begin{algorithm}
\label{algorithm: finding all maximal graphs in wHn(m)(p)(q)(n) with independence number equal to 2}
A modification of Algorithm \ref{algorithm: finding all maximal graphs in wHn(m)(p)(q)(n) with independence number at least k} for finding all maximal graphs in $\wHn{m}{p}{q}{n}$ with independence number 2 by adding $2$ independent vertices to the $(+K_{q - 1})$-graphs in $\wHn{(m - 1)}{p}{q}{n - 2}$ with independence number not greater than 2.\\
	
In step 1 of Algorithm \ref{algorithm: finding all maximal graphs in wHn(m)(p)(q)(n) with independence number at least k} we add the condition that the set $\mathcal{A}$ contains only the $(+K_{q - 1})$-graphs $\wHn{(m - 1)}{p}{q}{n - k}$ with independence number not greater than 2, and at the end of step 2.2 after the condition $\omega(G + e) = q, \forall e \in \E(\overline{G})$ we also add the condition $\alpha(G) = 2$.  
\end{algorithm}

Thus, finding all maximal graphs in $\wHn{m}{p}{q}{n}$ with independence number 2 is reduced to finding all $(+K_{q - 1})$-graphs with independence number not greater than 2 in $\wHn{m - 1}{p}{q}{n - 2}$ and finding the remaining maximal graphs in $\wHn{m}{p}{q}{n}$ with independence number greater than or equal to 3 is reduced to finding all $(+K_{q - 1})$-graphs in $\wHn{m - 1}{p}{q}{n - 3}$. In this way we can obtain all maximal graphs in $\wHn{m}{p}{q}{n}$ in steps, starting from graphs with a small number of vertices.\\

The \emph{nauty} programs \cite{McK90} have an important role in this work. We use them for fast generation of non-isomorphic graphs and for graph isomorph rejection.

\section{Computation of the number $\wFv{8}{6}{7}$}

\begin{figure}
	\centering
	\begin{subfigure}[b]{0.5\textwidth}
		\centering
		\includegraphics[height=240px,width=120px]{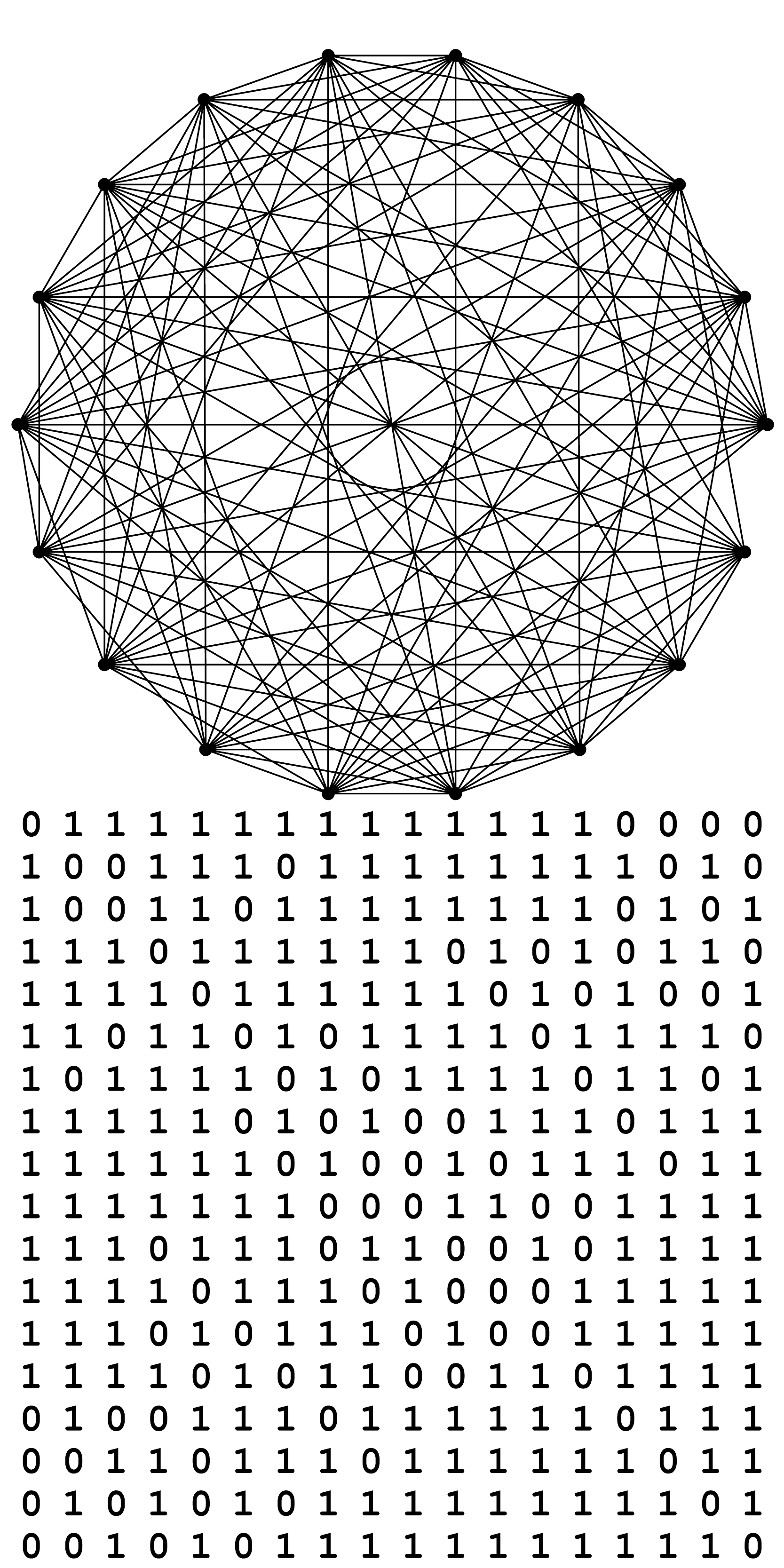}
		\caption*{\emph{$\Gamma_1$}}
		\label{figure: Gamma_1}
	\end{subfigure}
	\caption{Graph $\Gamma_1 \in \wHn{8}{6}{7}{18}$}
	\label{figure: wHn(8)(6)(7)(18)}
\end{figure}

From Theorem \ref{theorem: m_0} it becomes clear that in order to compute the numbers $\wFv{m}{6}{m - 1}$ we need the exact value of the number $m_0(6)$. According to Theorem \ref{theorem: m_0} (c), to obtain an upper bound for this number we need to know $\wFv{8}{6}{7}$. In this section we compute this number by proving the following
\begin{theorem}
\label{theorem: wFv(8)(6)(7) = 18}
$\wFv{8}{6}{7} = 18$.
\end{theorem}
\begin{proof}
The inequality $\wFv{8}{6}{7} \leq 18$ is proved in \cite{BN15} with the help of the graph $\Gamma_1$ which is given on Figure \ref{figure: wHn(8)(6)(7)(18)}(see the proof of Theorem 1.10 in version 1 or the proof of Theorem 1.9 in version 2). To obtain the lower bound we will prove with the help of a computer that $\wHn{8}{6}{7}{17} = \emptyset$.

First, we search for maximal graphs in $\wHn{8}{6}{7}{17}$ with independence number greater than 2. It is clear that $K_6$ and $K_6 - e$ are the only $(+K_6)$-graphs in $\wHn{3}{6}{7}{6}$. With the help of Algorithm \ref{algorithm: finding all maximal graphs in wHn(m)(p)(q)(n) with independence number at least k} we add 2 independent vertices to these graphs to find all maximal graphs in $\wHn{4}{6}{7}{8}$. By removing edges from them we find all $(+K_6)$-graphs in $\wHn{4}{6}{7}{8}$. In the same way, we successively obtain all maximal and all $(+K_6)$-graphs in the sets:\\
$\wHn{5}{6}{7}{10}$,  $\wHn{6}{6}{7}{12}$, $\wHn{7}{6}{7}{14}$.\\
In the end, with the help of Algorithm \ref{algorithm: finding all maximal graphs in wHn(m)(p)(q)(n) with independence number at least k} we add 3 independent vertices to the obtained $(+K_6)$-graphs in $\wHn{7}{6}{7}{14}$ to find all maximal graphs in $\wHn{8}{6}{7}{17}$ with independence number greater than 2.

After that, we search for maximal graphs in $\wHn{8}{6}{7}{17}$ with independence number 2. It is clear that $K_5$ is the only $(+K_6)$-graph in $\wHn{2}{6}{7}{5}$. With the help of Algorithm \ref{algorithm: finding all maximal graphs in wHn(m)(p)(q)(n) with independence number equal to 2} we add 2 independent vertices this graph to find all maximal graphs in $\wHn{3}{6}{7}{7}$ with independence number 2. By removing edges from them we find all $(+K_6)$-graphs in $\wHn{3}{6}{7}{7}$ with independence number 2. In the same way, we successively obtain all maximal and all $(+K_6)$-graphs with independence number 2 in the sets:\\
$\wHn{4}{6}{7}{9}$, $\wHn{5}{6}{7}{11}$, $\wHn{6}{6}{7}{13}$, $\wHn{7}{6}{7}{15}$ and $\wHn{8}{6}{7}{17}$.

The number of graphs found in each step is described in Table 1 in \cite{}. In both cases we do not obtain any maximal graphs in $\wHn{8}{6}{7}{17}$, therefore $\wHn{8}{6}{7}{17} = \emptyset$.
\end{proof}

\begin{corollary}
\label{corollary: 8 leq m_0(6) leq 11}
$8 \leq m_0(6) \leq 11$
\end{corollary}

\begin{proof}
The inequality $m_0(6) \geq 8$ follows from the definition of $m_0$ and the upper bound follows from Theorem \ref{theorem: m_0} (c), $p = 6$.
\end{proof}

\section{Proof of the Main Theorem}

Since $\wFv{8}{6}{7} = 18$, according to Theorem \ref{theorem: m_0} (a) it is enough to prove $m_0(6) = 8$. According to Corollary \ref{corollary: 8 leq m_0(6) leq 11} this equality will be proved if we prove $\wFv{9}{6}{8} > 18$, $\wFv{10}{6}{9} > 19$ and $\wFv{11}{6}{10} > 20$. The proof of these inequalities is similar to the proof of $\wFv{8}{6}{7} > 17$ from Theorem \ref{theorem: wFv(8)(6)(7) = 18}. It is clear that it is enough to prove $\wHn{m}{6}{m - 1}{m + 9} = \emptyset$ for $m = 9, 10, 11$.

First, we search for maximal graphs in $\wHn{m}{6}{m - 1}{m + 9}$ with independence number greater than 2. It is clear that $K_{m - 2}$ and $K_{m - 2} - e$ are the only $(+K_{m - 2})$-graphs in $\wHn{(m - 5)}{6}{m - 1}{m - 2}$. With the help of Algorithm \ref{algorithm: finding all maximal graphs in wHn(m)(p)(q)(n) with independence number at least k} we successively obtain all maximal and all $(+K_{m - 2})$-graphs in the sets:\\
$\wHn{(m - 4)}{6}{m - 1}{m}$\\
$\wHn{(m - 3)}{6}{m - 1}{m + 2}$\\
$\wHn{(m - 2)}{6}{m - 1}{m + 4}$\\
$\wHn{(m - 1)}{6}{m - 1}{m + 6}$\\
In the end, with the help of Algorithm \ref{algorithm: finding all maximal graphs in wHn(m)(p)(q)(n) with independence number at least k} we add 3 independent vertices to the obtained $(+K_{m - 2})$-graphs in $\wHn{(m - 1)}{6}{m - 1}{m + 6}$ to find all maximal graphs in $\wHn{m}{6}{m - 1}{m + 9}$ with independence number greater than 2.

After that, we search for maximal graphs in $\wHn{m}{6}{m - 1}{m + 9}$ with independence number 2. It is clear that $K_{m - 3}$ is the only $(+K_{m - 2})$-graph in $\wHn{(m - 6)}{6}{m - 1}{m - 3}$. With the help of Algorithm \ref{algorithm: finding all maximal graphs in wHn(m)(p)(q)(n) with independence number equal to 2} we successively obtain all maximal and all $(+K_{m - 2})$-graphs with independence number 2 in the sets:\\
$\wHn{(m - 5)}{6}{m - 1}{m - 1}$\\
$\wHn{(m - 4)}{6}{m - 1}{m + 1}$\\
$\wHn{(m - 3)}{6}{m - 1}{m + 3}$\\
$\wHn{(m - 2)}{6}{m - 1}{m + 5}$\\
$\wHn{(m - 1)}{6}{m - 1}{m + 7}$\\
$\wHn{m}{6}{m - 1}{m + 9}$.

The number of graphs found in each step is given in Table 2, Table 3 and Table 4 in \cite{}. In both cases we do not obtain any maximal graphs in the sets $\wHn{m}{6}{m - 1}{m + 9}$, $m = 9, 10, 11$, hence it follows $\wFv{9}{6}{8} > 18$, $\wFv{10}{6}{9} > 19$, $\wFv{11}{6}{10} > 20$ and $m_0(6) = 8$. Thus we finish the proof of the Main Theorem.


\bibliographystyle{plain}

\bibliography{main}


\author{Aleksandar Bikov}

\email{asbikov@fmi.uni-sofia.bg}, Corresponding author

\author{Nedyalko Nenov}

\email{nenov@fmi.uni-sofia.bg}

\smallskip

\address{
Faculty of Mathematics and Informatics

Sofia University "St. Kliment Ohridski

5, James Bourchier Blvd.

1164 Sofia, Bulgaria
}


\clearpage

\appendix

\section{Results of the computations}

\begin{table}[h]
	\centering
	\begin{tabular}{| p{3cm} | p{2cm} | p{2cm} | p{2cm} |}
		\hline
		set						& independence number	& maximal graphs 		& $(+K_6)$-graphs	\\
		\hline
		$\wHn{3}{6}{7}{6}$		& -						& 						& 2					\\
		$\wHn{4}{6}{7}{8}$		& -						& 2						& 13				\\
		$\wHn{5}{6}{7}{10}$		& -						& 8						& 324				\\
		$\wHn{6}{6}{7}{12}$		& -						& 56					& 104 271			\\
		$\wHn{7}{6}{7}{14}$		& -						& 18					& 1825				\\
		$\wHn{8}{6}{7}{17}$		& $\geq$ 3				& 0						&					\\
		\hline
		$\wHn{2}{6}{7}{5}$		& $\leq$ 2				& 						& 1					\\
		$\wHn{3}{6}{7}{7}$		& = 2					& 1						& 3					\\
		$\wHn{4}{6}{7}{9}$		& = 2					& 2						& 22				\\
		$\wHn{5}{6}{7}{11}$		& = 2					& 5						& 468				\\
		$\wHn{6}{6}{7}{13}$		& = 2					& 24					& 97 028			\\
		$\wHn{7}{6}{7}{15}$		& = 2					& 468					& 2 395 573			\\
		$\wHn{8}{6}{7}{17}$		& = 2					& 0						&					\\
		\hline
		$\wHn{8}{6}{7}{17}$		& -						& 0						&					\\
		\hline
	\end{tabular}
	\caption{Steps in the search of all maximal graphs in $\wHn{8}{6}{7}{17}$}
	\label{table: finding all graphs in wHn(8)(6)(7)(17)}
	
	\vspace{2em}
	
	\centering
	\begin{tabular}{| p{3cm} | p{2cm} | p{2cm} | p{2cm} |}
		\hline
		set						& independence number	& maximal graphs 		& $(+K_7)$-graphs	\\
		\hline
		$\wHn{4}{6}{8}{7}$		& -						& 						& 2					\\
		$\wHn{5}{6}{8}{9}$		& -						& 2						& 13				\\
		$\wHn{6}{6}{8}{11}$		& -						& 8						& 326				\\
		$\wHn{7}{6}{8}{13}$		& -						& 56					& 105 125			\\
		$\wHn{8}{6}{8}{15}$		& -						& 20					& 1844				\\
		$\wHn{9}{6}{8}{18}$		& $\geq$ 3				& 0						&					\\
		\hline
		$\wHn{3}{6}{8}{6}$		& $\leq$ 2				& 						& 1					\\
		$\wHn{4}{6}{8}{8}$		& = 2					& 1						& 3					\\
		$\wHn{5}{6}{8}{10}$		& = 2					& 2						& 22				\\
		$\wHn{6}{6}{8}{12}$		& = 2					& 5						& 489				\\
		$\wHn{7}{6}{8}{14}$		& = 2					& 25					& 119 124			\\
		$\wHn{8}{6}{8}{16}$		& = 2					& 506					& 2 747 120			\\
		$\wHn{9}{6}{8}{18}$		& = 2					& 0						&					\\
		\hline
		$\wHn{9}{6}{8}{18}$		& -						& 0						&					\\
		\hline
	\end{tabular}
	\caption{Steps in the search of all maximal graphs in $\wHn{9}{6}{8}{18}$}
	\label{table: finding all graphs in wHn(9)(6)(8)(18)}
\end{table}

\clearpage

\begin{table}[h]
	
	\vspace{3em}

	\centering
	\begin{tabular}{| p{3cm} | p{2cm} | p{2cm} | p{2cm} |}
		\hline
		set						& independence number	& maximal graphs 		& $(+K_8)$-graphs	\\
		\hline
		$\wHn{5}{6}{9}{8}$		& -						& 						& 2					\\
		$\wHn{6}{6}{9}{10}$		& -						& 2						& 13				\\
		$\wHn{7}{6}{9}{12}$		& -						& 8						& 327				\\
		$\wHn{8}{6}{9}{14}$		& -						& 56					& 105 281			\\
		$\wHn{9}{6}{9}{16}$		& -						& 20					& 1845				\\
		$\wHn{10}{6}{9}{19}$	& $\geq$ 3				& 0						&					\\
		\hline
		$\wHn{4}{6}{9}{7}$		& $\leq$ 2				& 						& 1					\\
		$\wHn{5}{6}{9}{9}$		& = 2					& 1						& 3					\\
		$\wHn{6}{6}{9}{11}$		& = 2					& 2						& 22				\\
		$\wHn{7}{6}{9}{13}$		& = 2					& 5						& 496				\\
		$\wHn{8}{6}{9}{15}$		& = 2					& 25					& 121 498			\\
		$\wHn{9}{6}{9}{17}$		& = 2					& 509					& 2 749 155			\\
		$\wHn{10}{6}{9}{19}$	& = 2					& 0						&					\\
		\hline
		$\wHn{10}{6}{9}{19}$	& -						& 0						&					\\
		\hline
	\end{tabular}
	\caption{Steps in the search of all maximal graphs in $\wHn{10}{6}{9}{19}$}
	\label{table: finding all graphs in wHn(10)(6)(9)(19)}
	
	\vspace{2em}
	
	\centering
	\begin{tabular}{| p{3cm} | p{2cm} | p{2cm} | p{2cm} |}
		\hline
		set						& independence number	& maximal graphs 		& $(+K_9)$-graphs	\\
		\hline
		$\wHn{6}{6}{10}{9}$		& -						& 						& 2					\\
		$\wHn{7}{6}{10}{11}$	& -						& 2						& 13				\\
		$\wHn{8}{6}{10}{13}$	& -						& 8						& 327				\\
		$\wHn{9}{6}{10}{15}$	& -						& 56					& 105 314			\\
		$\wHn{10}{6}{10}{17}$	& -						& 20					& 1845				\\
		$\wHn{11}{6}{10}{20}$	& $\geq$ 3				& 0						&					\\
		\hline
		$\wHn{5}{6}{10}{8}$		& $\leq$ 2				& 						& 1					\\
		$\wHn{6}{6}{10}{10}$	& = 2					& 1						& 3					\\
		$\wHn{7}{6}{10}{12}$	& = 2					& 2						& 22				\\
		$\wHn{8}{6}{10}{14}$	& = 2					& 5						& 498				\\
		$\wHn{9}{6}{10}{16}$	& = 2					& 25					& 121 863			\\
		$\wHn{10}{6}{10}{18}$	& = 2					& 509					& 2 749 171			\\
		$\wHn{11}{6}{10}{20}$	& = 2					& 0						&					\\
		\hline
		$\wHn{11}{6}{10}{20}$	& -						& 0						&					\\
		\hline
	\end{tabular}
	\caption{Steps in the search of all maximal graphs in $\wHn{11}{6}{10}{20}$}
	\label{table: finding all graphs in wHn(11)(6)(10)(20)}
\end{table}


\end{document}